\documentclass[10pt,a4paper]{article}
\usepackage[latin1]{inputenc}
\usepackage{amsmath, amsthm}

\usepackage{amsfonts}
\usepackage{amssymb}
\usepackage{graphicx}
\usepackage{hyperref}
\author{Tulasi Ram Reddy
	\footnote{Research supported in part by UGC (under SAP-DSA Phase IV). Research supported in part by ISF-UGC post-doctoral fellowship. Contents of this article are based on the author's PhD Thesis \cite{atrthesis}.}\\ Indian Statistical Institute, Bangalore\\ tulasi\_vs@isibang.ac.in
	}

\title{Probability that product of real random matrices have all
	eigenvalues real tend to 1.}

\date{}

	\newtheorem{theorem}{Theorem}[section]
	\newtheorem{lemma}[theorem]{Lemma}

	\numberwithin{equation}{section}
	
	\newtheorem{proposition}[theorem]{Proposition}
	\newtheorem{corollary}[theorem]{Corollary}
	\newtheorem{conjecture}{Conjecture}

\begin{document}

				\maketitle
		
		\begin{abstract}
			In this article we consider products of real random matrices with fixed size. Let $A_1,A_2, \dots $ be i.i.d $k \times k$ real matrices, whose entries are independent and identically distributed from probability measure $\mu$. Let $X_n = A_1A_2\dots A_n$. Then it is conjectured that $$\mathbb{P}(X_n \text{ has all real eigenvalues}) \rightarrow 1 \text{ as } n \rightarrow \infty.$$ We show that the conjecture is true when $\mu$ has an atom. \\

		\end{abstract}

		\section*{\normalfont \hfil Introduction and results \hfil}

		The study of products of random matrices was initiated by Furstenberg and Kesten in \cite{furstenberg}, where they studied the Lyaponov exponents for the product of i.i.d random matrices of fixed size. Since then there has been considerable interest in the study of products of random matrices. For detailed results and applications we refer the reader to \cite{cohen1986random}, \cite{Bougerol1985} and \cite{hognas}. Not much is known in understanding the spectrum for products of real random matrices.
		
		In \cite{arul}, Lakshminarayan observed an interesting phenomenon for products of $k\times k$ i.i.d matrices with i.i.d real Gaussian entries in the context of quantum entanglement. Let $p_n^{(k)}$ be the probability that the product of $n$ such matrices has all real eigenvalues. Based on numerical computations, it was conjectured in \cite{arul} that $p^{(k)}_n$ increases to $1$ with the size of the product. 
		
		Forrester, in \cite{forrester}, proved the result by giving an explicit formula for $p_n^{(k)}$, from which it was deduced that this probability increases to $1$ exponentially. The same result was shown in \cite{nanda} following a different approach. We state the a generalization of this result as a conjecture below.
		
		\begin{conjecture}\label{con1}
			Let $X_1,X_2, \dots X_n$ be i.i.d matrices of size $k \times k$, whose entries are i.i.d real random variables distributed according to probability measure $\mu$ and $A_n=X_1X_2\dots X_n$. Then,
			$$\lim\limits_{n\rightarrow \infty}\mathbb{P}(A_n\text{ has all real eigenvalues})=1.$$
		\end{conjecture}
		
		Numerical simulations suggest that the above conjecture is true for any probability measure $\mu$. In \cite{kavita}, numerical evidence was shown for the above conjecture for various probability measures $\mu$, which have continuous density. In \cite{forresterandsantosh} a similar computation was done for products of truncated orthogonal matrices.
		
		In this short note we prove the conjecture for a special case when the probability measure $\mu$ has an atom i.e., there is a real number $x$ such that $\mu(\{x\})>0$. We state the result in Theorem \ref{rank1} and show the special case of the conjecture as a corollary.
		
		\begin{theorem} \label{rank1}
			Let $X_1,X_2, \dots X_n$ be i.i.d random matrices of size $k \times k$,  distributed according to a probability measure $\nu$ such that $\mathbb{P}(X_1 \text{ has rank } 1) >0$  and $A_n=X_1X_2\dots X_n$. Then, 
			\[ 
			\lim\limits_{n\rightarrow\infty}\mathbb{P}(A_n \text{ has all real real eigenvalues})=1.
			\]	
		\end{theorem}
		As a special case if we assume that the random matrices have all entries as i.i.d. random variables distributed according to probability measure $\mu$. If $\mu$ has an atom then such random matrices satisfy the above hypothesis and resolve the conjecture \ref{con1} when the measure $\mu$ has an atom. We state it as the following corollary.
		\begin{corollary}
			Let $X_1,X_2, \dots X_n$ be i.i.d matrices of size $k \times k$, whose entries are i.i.d real random variables distributed according to a probability measure $\mu$ which has an atom and $A_n=X_1X_2\dots X_n$.  Then, 
			\[ 
			\lim\limits_{n\rightarrow\infty}\mathbb{P}(A_n \text{ has all real real eigenvalues})=1.
			\]
			
		\end{corollary} \label{atom}
		Any probability measure $\mu$ with discrete support satisfy the hypothesis of the above theorem. For example if the entries of the matrices are Rademacher distributed random variables (takes values $\pm1$ with equal probability), then the theorem asserts that product of such matrices have purely real spectrum with exponentially high probability. 

		The proof of the Theorem \ref{rank1} is elementary and based on a simple observation that rank of product of matrices is at most the minimum of the ranks of the individual matrices. In the given scenario, each individual matrix will be of rank at most $1$ with non zero probability. If a real matrix has rank at most $1$, then it has all real eigenvalues (they are $0$ and the trace of the matrix). 
		
		\begin{proof}[Proof of Theorem \ref{rank1}]
			\begin{align}
				\mathbb{P}(A_n \text{ has rank at most } 1 ) & \geq \mathbb{P}(\text{at least one of }X_1,X_2,\dots,X_n \text{ has rank at most }1), \nonumber \\
				& \geq 1-(1-\mathbb{P}(X_1 \text{ has rank }1))^n.\label{eqn:chapter4:lemma1:1}
			\end{align}
			
			We know that real matrices with rank at most $1$, have all eigenvalues real. Hence,
			$$
			\mathbb{P} (A_n \text{ has all real eigenvalues}) \geq \mathbb{P}(A_n \text{ has rank at most }1).
			$$ 	
			Therefore from above and \eqref{eqn:chapter4:lemma1:1}	we have,  $$\lim\limits_{n\rightarrow\infty}\mathbb{P}(A_n \text{ has all real eigenvalues})=1.$$
		\end{proof}

		We now show a lower bound, that is away from $0$, for the probability that the product of $2\times2$ i.i.d real random matrices having both the real eigenvalues.
		
		\begin{proposition}\label{2by2}
			Let $X_1,X_2, \dots X_n$ are i.i.d matrices of size $2 \times 2$ whose entries are i.i.d real random variables distributed according to probability measure $\mu$ and $A_n=X_1X_2\dots X_n$. Then,	
			$$
			\mathbb{P}(A_n\text{ has all real eigenvalues})\geq \frac{1}{2}.
			$$
		\end{proposition}
		
		Notice that the rows of product of i.i.d random matrices are exchangeable. We show the bound for the matrices whose rows are exchangeable in the following Lemma \ref{exchangeable}. Hence the proof of Proposition \ref{2by2} follows from the Lemma \ref{exchangeable}.
		
		\begin{lemma}\label{exchangeable}
			Let $M = \left[\begin{smallmatrix} a&b\\ c&d \end{smallmatrix}\right]$, where $(a,b)$ and $(c,d)$ are real exchangeable random variables. Then,
			$$
			\mathbb{P}(M \text{ has both real eigenvalues}) \geq \frac{1}{2}.
			$$
		\end{lemma}
		\begin{proof}
			
			The characteristic polynomial of the matrix $M$ is $P_M(x)=x^2-(a+d)x+(ad-bc)$. The matrix $M$ has all real eigenvalues if and only if the discriminant of the characteristic polynomial,
			$$(a+d)^2-4(ad-bc) \geq 0.$$
			Because $(a,c)$ and $(b,d)$ are exchangeable we have,
			$$
			\mathbb{P}((a+d)^2-4(ad-bc) \geq 0)=\mathbb{P}((b+c)^2-4(bc-bd) \geq 0).
			$$ 
			Therefore,
			\begin{align}
				\mathbb{P}(M \text{ has both} &\text{ real eigenvalues}) \nonumber\\ &= \frac{1}{2}(\mathbb{P}((a+d)^2-4(ad-bc) \geq 0)+\mathbb{P}((b+c)^2-4(bc-bd) \geq 0)),\nonumber\\
				& \geq \frac{1}{2}\mathbb{P}((a+d)^2-4(ad-bc) \geq 0 \text{ or }(b+c)^2-4(bc-bd) \geq 0). \label{eqn:chapter4:lemma2:1}
			\end{align} 
			Because $(a+d)^2-4(ad-bc)+(b+c)^2-4(bc-bd)\geq 0$, at least one of $(a+d)^2-4(ad-bc)$ and $(b+c)^2-4(bc-bd)$ is non-negative. Therefore,	 
			$$
			\mathbb{P}((a+d)^2-4(ad-bc) \geq 0 \text{ or }(b+c)^2-4(bc-bd) \geq 0)=1.
			$$
			Combining above and \eqref{eqn:chapter4:lemma2:1} we have that,
			$\mathbb{P}(M \text{ has both real eigenvalues}) \geq \frac{1}{2}.$
		\end{proof}
		
		\section*{\normalfont \hfil Acknowledgements \hfil}
		The author would like to thank M. Krishnapur for directing me towards this problem and having useful discussions on this problem. He would also like to thank S. Athreya for giving valuable feedback on this draft.
		\bibliographystyle{amsplain}
		\bibliography{paper}
\end{document}